\documentclass[11pt]{amsart}

\usepackage{amssymb}
\usepackage{enumitem}
\usepackage{xcolor}

\usepackage{graphicx}

\usepackage{url}
\makeatletter
\@namedef{subjclassname@2020}{%
  \textup{2020} Mathematics Subject Classification}
\makeatother

\usepackage[T1]{fontenc}

\newtheorem{theorem}{Theorem}[section]

\newtheorem{lemma}[theorem]{Lemma}

\theoremstyle{definition}

\newtheorem*{theorem*}{Theorem}

\numberwithin{equation}{section}

\def\R{{\mathbb R}}
\def\e{{\rm e}}

\def\d{{\,\rm d}}
\begin{document}

\baselineskip=16.99pt

%%%%%%%%%%%%%%%%

\title[]{A sharp Harnack bound for a nonlocal heat equation}
\author[M. Dembny]{Mateusz Dembny}
\address{
\begin{itemize}[label={}]
\item Faculty of Mathematics, Mechanics and Informatics\\ University of Warsaw\\
ul.\,Banacha 2, 02-097 Warsaw, Poland
\end{itemize}
}
\email{m.dembny@student.uw.edu.pl}

\author[M. Sierżęga]{Mikołaj Sierżęga}

\address{\begin{itemize}[label={}]
\item Faculty of Mathematics, Mechanics and Informatics\\ 
University of Warsaw\\
ul.\,Banacha 2, 02-097 Warsaw, Poland
\item Department of Mathematics Cornell University,
583 Malott Hall Ithaca, NY 14853 USA
\end{itemize}
}
\email{m.sierzega@uw.edu.pl, ms3427@cornell.edu}

\date{}

\begin{abstract}
A sharp double-sided Harnack bound is derived for positive solutions of a fractional order heat equation. 
\end{abstract}



\maketitle

\section{Introduction}

Consider the classical linear heat equation, $\partial_t w-\partial_{xx}^2 w=0$, posed in an infinite strip $S_T=\R\times (0,T)$. If we restrict our attention to smooth positive solutions, then the following important lower bound may be deduced: 
\begin{equation}\label{LY}
    \partial_t\ln w-|\partial_x \ln w|^2\geq -\frac{1}{2 t}\quad \mbox{ in }\quad S_T.
\end{equation}
This result is an instance of a family of estimates derived by Aronson and B\'enilan to tackle the problem of regularity of solutions of the porous medium equation \cite{ARONSON_BENILAN_1979}. A highly consequential generalisation to the setting of Riemannian manifolds, due to Li and Yau \cite{LI_YAU_1986}, resulted in \eqref{LY} being commonly associated with their names in the literature. 

Note, that inequality \eqref{LY} does not mention the initial moment and applies to \emph{all positive smooth solutions} regardless of their origin. In particular, no further assumptions on the asymptotic behaviour of solutions are required. On the contrary, it is this remarkable generality of the Aronson-B\'enilan-Li-Yau (ABLY) bound that imposes limitations on the spatial growth of solutions and the nature of the initial object. 
It should be stressed that \eqref{LY} is \emph{sharp} - a rare and desirable property in the field of analysis of partial differential equations. The bound is identically satisfied when evaluated on the fundamental solution, 
\[
g(x-x_*,t):=\frac{1}{\sqrt{4\pi t}}\e^{-\frac{|x-x_*|^2}{4 t}},
\]
\emph{irrespective} of where the location $x_*$ of the source is  at the initial instant. 
Estimate \eqref{LY} is also called a \emph{differential Harnack bound} for the linear heat flow. In line with this terminology, \eqref{LY} may be appropriately integrated (see \cite{HAMILTON_2011,LI_YAU_1986}), to reveal the well-known classical parabolic Harnack estimate, due independently to Hadamard and Pini \cite{HADAMARD_1954,PINI_1954},  
\begin{equation}\label{HP}
    \frac{w(x_2,t_2)}{w(x_1,t_1)}\geq \sqrt{\frac{t_1}{t_2}}\e^{-\frac{|x_2-x_1|^2}{4(t_2-t_1)}}, 
\end{equation}
where $x_1$ and $x_2$ are arbitrary and  $0<t_1<t_2<T\leq \infty$.

The differential bound \eqref{LY} is \emph{sharper} than the integrated one \eqref{HP}, as the latter reduces to identity only for $u(x,t)=g(x-x_*,t)$, with the initial mass concentrated at a particular point, related to $(x_i,t_i)$ through
\[
x_*=\frac{x_1 t_2-x_2 t_1}{t_2-t_1}.
\] 

Observe moreover, that for the bound to be useful, there needs to be a nonzero gap between $t_1$ and $t_2$, i.e., we cannot compare the solution at two different points in space at the same instant of time. This is a characteristic feature of parabolic Harnack bounds, see \cite{MOSER1964}. 

The following note is inspired by a natural problem, put forward in \cite{GAROFALO_2019}, of finding an appropriate extension for the ABLY inequality  to the context of nonlocal diffusion and in particular to the canonical model of the fractional heat equation 
\[
    \partial_t u+(-\Delta)^\frac{\alpha}{2}u=0\quad \mbox{on}\quad S_T
\]
with $0<\alpha<2$.  There are numerous definitions of the fractional power of the Laplace operator, see e.g.\,\cite{KWASNICKI_2017}, none of which will be employed in any of the considerations below.  However, to fix ideas, we may opt for the standard potential-theoretic definition: 
 \[
    (-\Delta)^\frac{\alpha}{2}f(x):=C(n,\alpha)P.V.\int_{\R}\frac{f(x)-f(y)}{|x-y|^{n+\alpha}}\d y,
\]
where 
\[
C(n,\alpha):=\frac{2^\alpha\Gamma\left(\frac{n+\alpha}{2}\right)}{\pi^\frac{n}{2}\left|\Gamma\left(-\frac{\alpha}{2}\right)\right|}.
\]
For the purpose of this note we will restrict our attention to one space dimension and $\alpha=1$, i.e.,the \emph{half-laplacian}, whereby 
\begin{equation}\label{FH}
\partial_t u+(-\Delta )^\frac{1}{2}u=0 \quad \mbox{ on }\quad S_T,
\end{equation}
with 
\[
    (-\Delta)^\frac{1}{2}f(x)=-\frac{1}{\pi}P.V.\int_{\R}\frac{f(x)-f(y)}{|x-y|^{2}}\d y.
\]

Before we begin, let us define the notion of a strong solution, used in the considerations that follow. Following \cite{BARRIOS_ET_AL_2014}, we will say that $u(x,t)$ is a \emph{strong solution} of the fractional heat equation \eqref{FH} in the strip $S_T$ if 
\begin{itemize}
    \item $\partial_t u\in C(\R\times (0,T))$,
    \item $ u\in C(\R\times [0,T))$,
    \item the equation \eqref{FH} is satisfied pointwise for every $(x,t)\in S_T$. 
\end{itemize}
The same definition, with obvious modifications, applies to the notion of the strong solution for the classical heat equation.

One promising line of inquiry is to find a fractional counterpart of the ABLY inequality. Note, that 
\eqref{LY} is equivalent to 
\[
\partial_{xx}^2 \ln w\geq -\frac{1}{2 t}
\]
in $S_T$. This form of the inequality has been generalised to the fractional setting in \cite{WEBER_ZACHER_2022} where, among other things, the following elegant inequality has been been proven.

\begin{theorem}[Thm.\,3.2  + Prop.\,3.3 in \cite{WEBER_ZACHER_2022}]
Let $u:S_T\mapsto (0,\infty)$ be a strong solution to the fractional heat equation \eqref{FH}. Then, the Li-Yau type inequality  
\[
    -(-\Delta )^\frac{1}{2}\ln u\geq -\frac{1}{ 2 t}
\]
holds in $S_T$. 
\end{theorem}

This inequality may be then employed to derive a Harnack bound. 
\begin{theorem}[Thm.\,5.2 in \cite{WEBER_ZACHER_2022}]
Let $0<t_1<t_2<\infty$ and $x_1,x_2\in \R$. If $u$ is a strong positive solution of \eqref{FH} on $\R\times [0,\infty)$, then 
\begin{equation}\label{ZW2}
    \frac{u(x_2,t_2)}{u(x_1,t_1)}\geq \sqrt{\frac{t_1}{t_2}}\e^{-C\left[1+\frac{|x_2-x_1|^2}{(t_2-t_1)^2}\right]},
\end{equation}
for some positive constant $C$.
\end{theorem}
For a more general statement and further results consult  \cite{WEBER_ZACHER_2022}. In that work, the authors explain how their bound differs from the Hadamard-Pini estimate. In particular, \eqref{ZW2} does not reduce to identity when applied to the fractional heat kernel with an appropriate choice of $(x_i,t_i)$. Moreover, due to the polynomial behaviour of the heat kernel, it is expected that the sharp Harnack bound would display a polynomial rather than exponential decay. Lastly, it would be desirable not to  assume continuity of the solution at the initial time, thus allowing for generalised initial conditions.

%Notwithstanding these developments, the problem of the \emph{fractional Li-Yau} estimate still poses a challenge. The method of proof laid out in \cite{LI_YAU_1986} relies on the maximum principle applied to a parabolic inequality derived for an auxiliary function $\Delta \ln w$, sometimes called within this context the \emph{Harnack quantity}. This approach seems to be ill-suited for the non-local case and it is not even clear what the correct form of the non-local Harnack quantity would be.

Another result in this direction is provided in \cite{BONFORTE_ET_AL_2017}. There the authors consider a weaker class of solutions, the very weak solutions. If we assume, that the solution is also a strong one, then we obtain the following result in our setting.   
\begin{theorem}[Thm.\,8.2 in \cite{BONFORTE_ET_AL_2017}]
Let $u$ be a positive strong solution of \eqref{FH} for $t>0$. Suppose moreover, that the initial condition is dominated by the fractional heat kernel away from the origin, in the following sense: 
\[
0\leq u_0(x)\leq \frac{1}{1+|x|^2},
\]
for $|x|\geq R_0\geq 0$. Then, for all $x_1,x_2\in \R$ and $t_1,t_2>0$, we have
\begin{equation}\label{BSV}
    \left(\frac{t_1}{t_2}\right)C_*\leq\frac{u(x_2,t_2)}{u(x_1,t_1)}\leq \left(\frac{t_1}{t_2}\right) C^*,
\end{equation}
where
\[
C_*:=C(R_0)\left(\frac{t_2}{t_1}\right)^2\left[1+\frac{\sqrt{|t_2-t_1|}+\left||x_1|^2-|x_2|^2\right|}{\sqrt{t_1}+|x_1|^2}\right]
\]
and
\[
C^*:=C(R_0)\left(\frac{t_2}{t_1}\right)^2\left[1+\frac{\sqrt{|t_2-t_1|}+\left||x_1|^2-|x_2|^2\right|}{\sqrt{t_2}+|x_2|^2}\right],
\]
for some constant $C$ dependent on $R_0$. 
\end{theorem}

The bound \eqref{BSV} does reflect the decay rate of the fractional heat kernel. Moreover, this bound, unlike the Hadamard-Pini estimate, is double-sided. This interesting feature is a result of nonlocality of the fractional flow. Here however, a constraint is placed on the initial condition. Since one of the uses of Harnack bounds is to obtain constraints on the initial data \cite{HAMILTON_2011}, it is desirable to have bounds, derivation of which avoids introducing restrictions on the initial datum.

%Faced with this conundrum, in this note we propose to bypass the problem by deriving the 'right' Harnack bound without relying on a non-local generalisation of the Li-Yau bound. We will do so by working with the integral representation of solutions.

Our contribution in this note concerns obtaining an \emph{unconditional} bound, which is to say, that apart from the solution being classical and positive, we do not impose any further restrictions on the spatial growth of solutions and of the initial data. 

\section{The Harnack bound}
In order to demonstrate, that no additional requirement is needed, we will refer to the fractional counterpart of the classical uniqueness theorem of Widder \cite{BARRIOS_ET_AL_2014,WIDDER_1944}. The approach presented herein applies both to the classical heat equation and its fractional counterpart.

Let us first reprove \eqref{HP} through the convolution formula for the heat equation. To begin with, suppose that $w$ is a nonnegative strong solution on $
\R\times [0,T)$, with initial condition $w_0$. By Widder's representation and uniqueness theorems \cite{WIDDER_1944}, the solution is unique and expressed by the integral  
\[
w(x,t)=\int_{\R}g(x-y,t)w_0(y)\d y, 
\]
with the kernel
\[
g(x,t)=\frac{1}{\sqrt{4 \pi t}}\e^{-\frac{|x|^2}{4 t}}.
\]
%Also, if  but for example a Dirac mass whereby following the customary abuse of notation we get that 
%\begin{equation}\label{ker}
%g(x-x_*,t)=\int_{\R}g(x-y,t)\delta_{x_*}(y)\d y
%\end{equation}
%for all $x,x_*\in \R$ and $t>0$. 

Since this kernel is strictly positive, nothing prevents us from performing an elementary estimate 
\begin{align*}
    w(x_2,t_2)&=\int_\R \left[\frac{g(x_2-y,t_2)}{g(x_1-y,t_1)}\right]g(x_1-y,t_1)w_0(y)\d y\\
    &\geq\inf_{y\in \R}\left[\frac{g(x_2-y,t_2)}{g(x_1-y,t_1)}\right]\int_\R g(x_1-y,t_1)w_0(y)\d y\\
    &= w(x_1,t_1)\inf_{y\in \R}\left[\frac{g(x_2-y,t_2)}{g(x_1-y,t_1)}\right].
\end{align*}
It is easy to convince oneself, that whenever $0<\sigma_1<\sigma_2<+\infty$, we have 
\[
\frac{|x_1-y|^2}{\sigma_1}-\frac{|x_2-y|^2}{\sigma_2}\geq -\frac{|x_2-x_1|^2}{\sigma_2-\sigma_1},
\]
which in turn,  when applied to the  heat kernel, yields 
\[
\e^{-\frac{|x_2|^2}{4\sigma_2}}\bigg/\e^{-\frac{|x_1|^2}{4\sigma_1}}\geq \e^{-\frac{|x_2-x_1|^2}{4(\sigma_2-\sigma_1)}},
\]
i.e.,
\begin{equation}\label{gauss}
\sqrt{\frac{\sigma_2}{\sigma_1}}\frac{g(x_2-y,\sigma_2)}{g(x_1-y,\sigma_1)}\geq \sqrt{4\pi (\sigma_2-\sigma_1)}g(x_2-x_1,\sigma_2-\sigma_1).
\end{equation}
This inequality is sharp and resolves into identity for $y=x_*$, with  
\[
x_*:=\frac{x_1\sigma_2-x_2\sigma_1}{\sigma_2-\sigma_1}.
\]
Hence, with $t_i=\sigma_i$, we find 
\[
w(x_2,t_2)\geq w(x_1,t_1)\sqrt{\frac{t_1}{t_2}}\sqrt{4\pi (t_2-t_1)}g(x_2-x_1,t_2-t_1),
\]
which is \eqref{HP}. 

Thus, we arrive at the desired classical estimate, with this difference however, that the initial data (and so the solution class covered) is restricted, as compared to the scope of the technique resting on the ABLY inequality.

We will now lift this restriction. Suppose $w$ is a smooth positive solution in $\R\times (0,T)$. Choose $0<\tau<t_1$ and consider $w^\tau$, defined as the restriction of $w$ to the subdomain $\R\times [\tau,T)$. Clearly, $w^\tau$ is a strong solution on its domain and by the Widder representation and uniqueness theorems we have
\[
w^\tau(x,t)=\int_{\R^n}g(x-y,t-\tau)w(y,\tau)\d y
\]
$w^\tau=w$ in $\R\times (\tau,T)$. We can now perform the same estimate as before but with $\sigma_i=t_i-\tau$ in place of $t_i$. In effect, we find 
\[
\frac{w(x_2,t_2)}{w(x_1,t_1)}=\frac{w^\tau(x_2,t_2)}{w^\tau(x_1,t_1)}\geq \sqrt{\frac{t_1-\tau}{t_2-\tau}}\sqrt{4\pi (t_2-t_1)}g(x_2-x_1,t_2-t_1).
\]
We are free to apply this procedure with any choice of $\tau\in (0,t_1)$ and so 
\begin{align*}
\frac{w(x_2,t_2)}{w(x_1,t_1)}&\geq \sqrt{4\pi (t_2-t_1)}g(x_2-x_1,t_2-t_1)\sup_{0<\tau<t_1}\sqrt{\frac{t_1-\tau}{t_2-\tau}}\\
&= \sqrt{\frac{t_1}{t_2}}\sqrt{4\pi (t_2-t_1)}g(x_2-x_1,t_2-t_1).
\end{align*}
Thus, provided the solution we work with is smooth and positive in $\R\times (0,T)$, we do not need to impose additional constraints on the  nature of the initial condition. 

We see then, that it is the Widder representation and uniqueness theorems that ensure, that our straightforward estimation catches the proper class of solutions, without unnecessary additional restrictions. To apply a similar reasoning to the fractional heat flow we need an appropriate generalisation of these theorems. Indeed, such a result is available \cite{BARRIOS_ET_AL_2014} and below we cite a version tailored to our needs. 

%We would like to follow the same scheme in the case of the fractional heat flow \eqref{FH}. To this end we need a fractional counterpart of the representation and uniqueness theorems. These are available and below we recall a  particular case of a Widder-type representation theorem for positive solutions of the fractional heat equations demonstrated in \cite{BARRIOS_ET_AL_2014}, see Theorem 1.4 of that work.
\begin{theorem}[Thm.\,1.4 + Thm.\,2.1 in \cite{BARRIOS_ET_AL_2014}]\label{rep}
Let $v$ be a nonnegative strong solution of the problem 
\[
\partial_t v+(-\Delta)^\frac{1}{2}v=0 \quad \mbox{in}\quad \R\times [\tau,T),
\]
with $v(\cdot,\tau)=\nu\in C(\R)$. Then, $v$ is unique and admits the representation 
\[
v(x,t)=\int_{\R}k(x-y,t-\tau)\nu(y)\d y,
\]
with 
\[
k(x,t):=\frac{1}{\pi t}\left(\frac{1}{1+\frac{|x|^2}{t^2}}\right).
\]
\end{theorem}
In particular, since the kernel is a smooth function, the solution needs to be smooth in $\R\times (\tau,T)$. 
Actually, in this one-dimensional context and for the square root of the Laplace operator, the above representation has been known to Widder even before his celebrated representation theorem for the classical heat flow \cite{LOOMIS_WIDDER_1942}.

First, we will derive a simple counterpart of \eqref{HP}, that relies on a lemma inspired by the appealing inequality \eqref{gauss}.

\begin{lemma}\label{bound}
Let $0<\sigma_1<\sigma_2<+\infty$ and $x_1,x_2\in \R$, then 
\[
    \left(\frac{\sigma_2}{\sigma_1}\right)\frac{k(x_2-y,\sigma_2)}{k(x_1-y,\sigma_1)}\geq \pi(\sigma_2-\sigma_1) k(x_2-x_1,\sigma_2-\sigma_1).
\]

\end{lemma}
\begin{proof}
Set $\lambda=\frac{|x_2-x_1|^2}{(\sigma_2-\sigma_1)^2}$. We need to show, that 
\[
\frac{1+\frac{|x_1-y|^2}{\sigma_1^2}}{1+\frac{|x_2-y|^2}{\sigma_2^2}}\geq\frac{1}{ 1+\frac{|x_2-x_1|^2}{(\sigma_2-\sigma_1)^2}},
\]
which is equivalent to 
\[
\sigma_2^2|x_1-y|^2(1+\lambda)-\sigma_1^2|x_2-y|^2+\lambda \sigma_1^2\sigma_2^2\geq 0.
\]
This in turn, may be rephrased as a quadratic inequality in $y$:  
\[
A y^2+ B y+C\geq 0
\]
with
\[
\begin{cases}
A=(1+\lambda)\sigma_2^2-\sigma_1^2, \\
B=2 \sigma_1^2 x_2-2 (1+\lambda)\sigma_2^2 x_1,\\ C=(1+\lambda)\sigma_2^2 x_1^2-\sigma_1^2 x_2^2+\lambda \sigma_1^2\sigma_2^2.
\end{cases}
\]
Since $\lambda\geq 0$ and $\sigma_2>\sigma_1$, we have $A>0$. Hence, it suffices to check, that the discriminant is not positive. Thus, after a brief calculation, this amounts to the requirement that:    
\[4   \sigma_1^2 \sigma_2^2(1+\lambda) \left|x_2-x_1\right|^2- 4 \lambda \sigma_1^2\sigma_2^2\left[\left(1+\lambda\right) \sigma_2^2 - \sigma_1^2  \right] \leq 0.
\]
Now, $\sigma_i\neq 0$ and by definition $|x_2-x_1|^2=\lambda(\sigma_2-\sigma_1)^2$. Hence, the above simplifies further to
\[\lambda\left[(1+\lambda) (2\sigma_1\sigma _2-\sigma_1^2)- \sigma_1^2  \right]  \geq 0,
\]
which is true, since $\lambda\geq 0$ and $\sigma_2>\sigma_1$. 
\end{proof}

\begin{theorem}\label{lowharnack}
Let $u$ be a smooth positive solution of the fractional heat equation \eqref{FH} in $\R\times (0,T)$. Then, given $0<t_1<t_2<T$ and $x_1,x_2\in \R$, we have 
\[
    \frac{u(x_2,t_2)}{u(x_1,t_1)}\geq \left(\frac{t_1}{t_2}\right)\frac{1}{1+\frac{|x_2-x_1|^2}{(t_2-t_1)^2}}.
\]
\end{theorem}
\begin{proof}
Choose $0<\tau<t_1$ and restrict $u$ to the $\R\times [\tau,T)$ subdomain, where it becomes a strong solution. By the Widder-type representation and uniqueness theorem \cite{BARRIOS_ET_AL_2014}, we may write 
\[
u(x,t)=\int_{\R}k(x-y,t-\tau)u(y,\tau)\d y.
\]
Further, due to Lemma \ref{bound} with $\sigma_i=t_i-\tau$, we have 

\begin{align*}
    u(x_2,t_2)&=\int_{\R}\left[\frac{k(x_2-y,t_2-\tau)}{k(x_1-y,t_1-\tau)}\right]k(x_1-y,t_1-\tau)u(y,\tau)\d y\\
    &\geq \left(\frac{t_1-\tau}{t_2-\tau}\right)\frac{1}{1+\frac{|x_2-x_1|^2}{(t_2-t_1)^2}}u(x_1,t_1).
\end{align*}
The above inequality is valid for all $\tau\in (0,t_1)$ and so we may optimise by taking $\tau$ arbitrarily small, to the effect that
\[
u(x_2,t_2)\geq \left(\frac{t_1}{t_2}\right)\frac{1}{1+\frac{|x_2-x_1|^2}{(t_2-t_1)^2}}u(x_1,t_1),
\]
as required. 
\end{proof}
The above theorem satisfies some of the expected properties of a Harnack bound for \eqref{FH} in that it is related to the fractional heat kernel and its decay properties. However, it is one-sided when a double-sided bound is expected and it does not reduce to identity when the solution is given by the fractional heat kernel, originating at some specific location related to $(x_i,t_i)$.  

Next, we move on to our main result, i.e., an optimal fractional counterpart of the Hadamard-Pini bound.

\begin{theorem} 
Let $u$ be a positive classical solution of \eqref{FH} on $\R\times(0,T)$.  
Given $0<t_1,t_2<T$ and $x_1,x_2\in \R$, we have
\begin{equation}\label{FLY}
 \left(\frac{t_1}{t_2}\right) C_*\leq\frac{u(x_2,t_2)}{u(x_1,t_1)}\leq \left(\frac{t_1}{t_2}\right)C^*,
\end{equation}
with 
\[
C_*=\frac{\sqrt{\kappa_0}-(t_2-t_1)(t_2+t_1)-|x_2-x_1|^2} {\sqrt{\kappa_0}-(t_2-t_1)(t_2+t_1)+|x_2-x_1|^2},
\]
\[
C^*=\frac{\sqrt{\kappa_0}+(t_2-t_1)(t_2+t_1)+|x_2-x_1|^2} {\sqrt{\kappa_0}+(t_2-t_1)(t_2+t_1)-|x_2-x_1|^2},
\]
where 
\[
\kappa_0=\Big(\big|x_2-x_1\big|^2+\big|t_2-t_1\big|^2\Big)\Big(\big|x_2-x_1\big|^2+\big|t_2+t_1\big|^2\Big).
\]
Moreover, 
\[
C_*=\frac{x_1-x_*}{x_2-x_*}\quad \mbox{ and }\quad C^*=\frac{x_1-x^*}{x_2-x^*},
\]
where $x_*,x^*\in \R$ satisfy
\[
\left(\frac{t_1}{t_2}\right)C_*=\frac{k(x_2-x_*,t_2)}{k(x_1-x_*,t_1)}\quad \mbox{ and }\quad \frac{k(x_2-x^*,t_2)}{k(x_1-x^*,t_1)}=\left(\frac{t_1}{t_2}\right)C^*.
\]
\end{theorem}

Before laying out the proof, it is worth stressing that, like in the estimate found in \cite{BONFORTE_ET_AL_2017}, the times $t_1$ and $t_2$ \emph{need not} be ordered or different. The formulae for $M_*$ and $M^*$ are presented in a possibly simple form that emphasises the spatial $|x_2-x_1|$ and temporal $|t_2-t_1|$ distance between the points $(x_1,t_1)$ and $(x_2,t_2)$. In the Hadamard-Pini  bound \eqref{HP},  the counterparts of $C_*$ and $C^*$ depend on the points $(x_i,t_i)$ through the distances $|x_2-x_1|$ and $|t_2-t_1|$. In the fractional case however our estimate is also sensitive to the life-span of the solution prior to the instants $t_1$ and $t_2$, which manifests itself through appearance of the term $(t_2+t_1)$ alongside $(t_2-t_1)$.

\begin{proof} 

Take $ 0<\tau<\min\{t_1,t_2\}$. As before, when considered on $\R\times [\tau,T)$, the solution $u$ is strong and due to the representation formula on $C(\R\times [\tau,T))$ we may write
\begin{align*}
   u(x_2,t_2)&=\int_{\R} k(x_2-y,t_2-\tau)u(y,\tau)\d y \\
   &=\int_\R\left[ \frac{k(x_2-y,t_2-\tau)}{k(x_1-y,t_1-\tau)}\right]k(x_1-y,t_1-\tau)u(y,\tau)\d y.
\end{align*}
Since $u$ is positive  
\[
m_*(\tau)u(x_1,t_1)\leq u(x_2,t_2)\leq m^*(\tau)u(x_1,t_1),
\]
where  
\[
m_*(\tau):=\inf_{y\in \R}\, \frac{k(x_2-y,t_2-\tau)}{k(x_1-y,t_1-\tau)} \quad \mbox{and} \quad m^*(\tau):=\sup_{y\in \R}\, \frac{k(x_2-y,t_2-\tau)}{k(x_1-y,t_1-\tau)}.
\]
If we now put 
\[
M_*:=\sup_{0<\tau<\min\{t_1,t_2\}}m_*(\tau)\quad\mbox{ and }\quad  M^*:=\inf_{0<\tau< \min\{t_1,t_2\}}m^*(\tau)
\]
then 
\[
 M_*\leq\frac{u(x_2,t_2)}{u(x_1,t_1)}\leq M^*,
\]
i.e., \eqref{FLY} with
\[
C_*:=\left(\frac{t_1}{t_2}\right)M_* \quad \mbox{ and }\quad  C^*=\left(\frac{t_1}{t_2}\right) M^*.
\]
It remains to show, that these bounds are finite and given as in the statement of the theorem. 

First, we will obtain an explicit expression for $m_*(\tau)$ and $m^*(\tau)$ and further show, that $M_*=m_*(0)$ and $M^*=m^*(0)$. 
In the calculations below, it will be expedient to introduce new coordinates:
\[
\omega(y):=\frac{1}{2}\left(\frac{x_2-y}{t_2-\tau}+\frac{x_1-y}{t_1-\tau}\right)\quad \mbox{and}\quad \rho(y):=\frac{1}{2}\left(\frac{x_2-y}{t_2-\tau}-\frac{x_1-y}{t_1-\tau}\right).
\]
We are interested in extrema of the function

\[
    \frac{k(x_2-y,t_2-\tau)}{k(x_1-y,t_1-\tau)}=\left(\frac{t_1-\tau}{t_2-\tau}\right)\frac{k\left(\frac{x_2-y}{t_2-\tau},1\right)}{k\left(\frac{x_1-y}{t_1-\tau},1\right)}=\left(\frac{t_1-\tau}{t_2-\tau}\right)\frac{1+|\omega-\rho|^2}{1+|\omega+\rho|^2},
\]
with respect to $y$. This function is smooth, not constant and approaches $\frac{t_2-\tau}{t_1-\tau}$ as $|y|\to +\infty$. Moreover, it achieves values both above and below $\frac{t_2-\tau}{t_1-\tau}$. Hence, there is a global  global minimum and a global maximum and we will compute them directly.
It is computationally more convenient to work with the function
\[
H(y):=\ln\left(\frac{k(x_2-y,t_2-\tau)}{k(x_1-y,t_1-\tau)}\right).
\]
We find that 
\[
    H'(y)=2\frac{\omega-\rho}{1+|\omega-\rho|^2}(\omega'-\rho')-2\frac{\omega+\rho}{1+|\omega+\rho|^2}(\omega'+\rho')=0,
\]
when
\begin{equation}\label{eq}
    2(\rho'\omega-\rho\omega')(\omega-\rho)(\omega+\rho)+2(\rho\omega)'=0.
\end{equation}
Now
\[
\omega-\rho=\frac{x_1-y}{t_1-\tau}, \quad \omega+\rho=\frac{x_2-y}{t_2-\tau}\quad  \mbox{and}\quad 4\rho\omega=\frac{|x_2-y|^2}{(t_2-\tau)^2}-\frac{|x_1-y|^2}{(t_1-\tau)^2}.
\]
Further,
\[
2(\rho'\omega-\rho\omega')=\frac{x_2-x_1}{(t_2-\tau)(t_1-\tau)}
\]
and
\[
2(\rho\omega)'=\frac{x_1-y}{(t_1-\tau)^2}-\frac{x_2-y}{(t_2-\tau)^2}.
\]
Hence, \eqref{eq} may be rephrased as
\[
\frac{(x_2-y)(x_1-y)(x_2-x_1)}{(t_2-\tau)^2(t_1-\tau)^2}+\frac{x_1-y}{(t_1-\tau)^2}-\frac{x_2-y}{(t_2-\tau)^2}=0,
\]
or 
\[
(x_2-y)(x_1-y)(x_2-x_1)+(x_1-y)(t_2-\tau)^2-(x_2-y)(t_1-\tau)^2=0.
\]
This, of course, is simply a quadratic equation in $y$:
\begin{equation}\label{qdr}
A y^2+B y+C=0,
\end{equation}
with 
\[
A=x_2-x_1, \quad B=-\left[x_2^2-x_1^2+(t_2-\tau)^2-(t_1-\tau)^2\right]
\]
and 
\[
C=x_1 x_2(x_2-x_1)+\left(t_2-{\tau}\right)^2x_1-\left(t_1-{\tau}\right)^2x_2.
\]
We see, that the discriminant,
\[
\kappa(\tau)=\Big(\big|x_2-x_1\big|^2+\big|(t_2-\tau)-(t_1-\tau)\big|^2\Big)\Big(\big|x_2-x_1\big|^2+\big|(t_2-\tau)+(t_1-{\tau})\big|^2\Big),
\]
is nonnegative and vanishes, if and only if both $x_1=x_2$ and $t_1=t_2$. If $x_1=x_2$, we simply get
\[
m_*(\tau)=m^*(\tau)=\frac{k(0,t_2-\tau)}{k(0,t_1-\tau)}=\frac{t_1-\tau}{t_2-\tau}. 
\]
Otherwise, let $x_*(\tau)$ and $x^*(\tau)$ be the roots of \eqref{qdr}:
\[
x_*(\tau):=\frac{x_2^2-x_1^2+(t_2-\tau)^2-(t_1-\tau)^2- \sqrt{\kappa(\tau)}}{2(x_2-x_1)},
\]
\[
x^*(\tau):=\frac{x_2^2-x_1^2+(t_2-\tau)^2-(t_1-\tau)^2+ \sqrt{\kappa(\tau)}}{2(x_2-x_1)}.
\]
Then, from \eqref{eq} we infer, that if $\hat y(\tau)$ is either of them, we have
\begin{align*}
    \frac{1+|w-\rho|^2}{1+|w+\rho|^2}&=\frac{(w-\rho)(w'-\rho')}{(w+\rho)(w'+\rho')}=\frac{\big((w-\rho)^2\big)'}{\big(|w+\rho|^2\big)'}\\
    &=\frac{\left(t_2-{\tau}\right)^2}{\left(t_1-{\tau}\right)^2}\left(\frac{ x_1-\hat y(\tau)}{x_2-\hat y(\tau)}\right).
\end{align*}
Further, 
\begin{align*}
   \frac{ x_1-x_*(\tau)}{x_2-x_*(\tau)}&=\frac{\sqrt{\kappa(\tau)}-\big[(t_2-\tau)-(t_1-\tau)\big]\big[(t_2-\tau)+(t_1-\tau)\big]-|x_2-x_1|^2} {\sqrt{\kappa(\tau)}-\big[(t_2-\tau)-(t_1-\tau)\big]\big[(t_2-\tau)+(t_1-\tau)\big]+|x_2-x_1|^2}
\end{align*}
and 
\begin{align*}
   \frac{ x_1-x^*(\tau)}{x_2-x^*(\tau)}&=\frac{\sqrt{\kappa(\tau)}+\big[(t_2-\tau)-(t_1-\tau)\big]\big[(t_2-\tau)+(t_1-\tau)\big]+|x_2-x_1|^2} {\sqrt{\kappa(\tau)}+\big[(t_2-\tau)-(t_1-\tau)\big]\big[(t_2-\tau)+(t_1-\tau)\big]-|x_2-x_1|^2}.
\end{align*}
Now, we may set 
\[
m_*(\tau)=\left(\frac{t_1-\tau}{t_2-\tau}\right)\frac{ x_1-x_*(\tau)}{x_2-x_*(\tau)}\quad \mbox{and}\quad m^*(\tau)=\left(\frac{t_1-\tau}{t_2-\tau}\right)\frac{ x_1-x^*(\tau)}{x_2-x^*(\tau)}.
\]
In order to see, that irrespective of $\tau$, we have $m_*(\tau)\leq m^*(\tau)$, we note, that given $\alpha$ and $\gamma$ nonnegative and $\beta\in \R$, such that $\alpha> |\beta|+\gamma$,  we have  
\begin{equation}\label{ineq}
\frac{\alpha+\beta+\gamma}{\alpha+\beta-\gamma}\geq \frac{\alpha-\beta-\gamma}{\alpha-\beta+\gamma}.
\end{equation}
This follows, since the assumptions on $\alpha,\beta$ and $\gamma$ guarantee, that both denominators in \eqref{ineq} are positive and so equivalently
\[
(\alpha+\beta+\gamma)(\alpha-\beta+\gamma)- (\alpha-\beta-\gamma)(\alpha+\beta-\gamma)\geq 0,
\]
which, when expanded, reduces to $4\alpha\gamma\geq 0$.

Set 
\[
\alpha=\sqrt{\kappa(\tau)}, \quad \beta=\big[(t_2-\tau)-(t_1-\tau)\big]\big[(t_2-\tau)+(t_1-\tau)\big]
\]
and 
\[
\gamma=|x_2-x_1|^2.
\]
We will now establish, that $\alpha>|\beta|+\gamma$. It suffices to show, that 
\[
\alpha^2-\left(|\beta|+\gamma\right)^2=\left(\alpha-|\beta|-\gamma\right)\left(\alpha+|\beta|+\gamma\right)>0,
\]
which, in our case, translates to 
\begin{align*}
\alpha^2&=\Big(\gamma+\big|(t_2-\tau)-(t_1-\tau)\big|^2\Big)\Big(\gamma+\big|(t_2-\tau)+(t_1-{\tau})\big|^2\Big)\\
&=\gamma^2+\gamma\left[|(t_2-\tau)-(t_1-\tau)|^2+|(t_2-\tau)+(t_1-{\tau})|^2\right]+\beta^2\\
&=\gamma^2+2\gamma\left[(t_2-\tau)^2+(t_1-\tau)^2\right]+|\beta|^2.
\end{align*}
We have 
\[
(t_2-\tau)^2+(t_1-\tau)^2= \left|(t_2-\tau)^2-(t_1-\tau)^2\right|+2\min\left\{(t_2-\tau)^2,(t_1-\tau)^2\right\}
\]
and so 
\[
\alpha^2=\left(\gamma+|\beta|\right)^2+2\min\left\{(t_2-\tau)^2,(t_1-\tau)^2\right\}>0.
\]
In effect, we conclude that 
\[
 m_*(\tau)\leq  \frac{k(x_2-y,t_2-\tau)}{k(x_1-y,t_1-\tau)}\leq m^*(\tau)
\]
for all $y\in\R$.
It is a matter of a tedious but elementary calculation to show that 
\[
\frac{\d}{\d\tau}m_*(\tau)=-\frac{\sqrt{\left|x_2-x_1\right|^2+\left(t_2-t_1\right)^2}\left((t_2-\tau)+(t_1-\tau)\right)}{\left(t_2-{\tau}\right)\left(t_1-{\tau}\right)\sqrt{\left|x_2-x_1\right|^2+\left((t_2-\tau)+(t_1-\tau)\right)^2}}\leq 0
\]
and likewise
\[
\frac{\d}{\d\tau}m^*(\tau)=\frac{\sqrt{\left|x_2-x_1\right|^2+\left(t_2-t_1\right)^2}\left((t_2-\tau)+(t_1-\tau)\right)}{\left(t_2-{\tau}\right)\left(t_1-{\tau}\right)\sqrt{\left|x_2-x_1\right|^2+\left((t_2-\tau)+(t_1-\tau)\right)^2}}\geq 0.
\]
From this, we infer that 
\[
M_*=\sup_{0<\tau<\min\{t_1,t_2\}}m_*(\tau)=\left(\frac{t_1}{t_2}\right)\frac{\sqrt{\kappa_0}-(t_2-t_1)(t_2+t_1)-|x_2-x_1|^2} {\sqrt{\kappa_0}-(t_2-t_1)(t_2+t_1)+|x_2-x_1|^2}
\]
and 
\[
M^*=\inf_{0<\tau<\min\{t_1,t_2\}}m^*(\tau)=\left(\frac{t_1}{t_2}\right)\frac{\sqrt{\kappa_0}+(t_2-t_1)(t_2+t_1)+|x_2-x_1|^2} {\sqrt{\kappa_0}+(t_2-t_1)(t_2+t_1)-|x_2-x_1|^2}
\]
with 
\[
\kappa_0=\Big(\big|x_2-x_1\big|^2+\big|t_2-t_1\big|^2\Big)\Big(\big|x_2-x_1\big|^2+\big|t_2+t_1\big|^2\Big)
\]
as required. 
\end{proof}

\section{Discussion}
One striking difference between the Hadamard-Pini bound \eqref{HP} and its fractional counterpart \eqref{FLY} is the double-sidedness of the latter. When perceived as a model of diffusion, that - morally speaking - should share some broad characteristics with the standard heat flow, it may appear surprising, that the fractional flow admits an upper bound. Upon closer inspection we see however, that replacing the Laplace operator with its fractional power radically constraints the space of initial data, or -- if we choose not to refer to the initial condition --  the admissible growth of solutions in the spatial  direction. This limitation is a necessary prerequisite of using $(-\Delta)^\frac{1}{2}$ in the first place. No such \emph{a priori} constraint is found in the local case. The relatively narrow domain of the fractional operator means, that some of the intended applications of the Harnack bound become irrelevant. For example the Aronson-B\'enilan-Li-Yau inequality may be used to characterise those smooth positive solutions of the heat equation, that exist on a strip $\R\times (0,T)$, leading in effect to Tikhonov type conditions, see \cite{HAMILTON_2011}. The necessity of such constraints stems from the fact that smooth positive solutions of the heat equation may blow-up in finite time, in case there is so much heat ``tucked away at space-infinity'', that the averaging process commanded by the diffusion operator cannot redistribute it efficiently enough and there comes a time when the solution becomes unbounded everywhere. Following this line of thought, we may use the Hadamard-Pini bound to characterise those initial conditions, that give rise to global-in-time solutions. Considerations of this type are unnecessary for the fractional heat equation \eqref{FH}, since \emph{all smooth positive solutions are global}, see \cite{BARRIOS_ET_AL_2014,VAZQUEZ_2018}. This circumstance is not limited to the particular instance of the fractional heat flow considered here but applies more broadly to other exponents and nonlocal operators. 

Existence of a double-sided bound adds another layer of difficulty to the problem of finding a proper extension of differential Harnack bounds of Aronson-B\'enilan-Li-Yau type. In the case of the standard heat flow, the transition from the differential Harnack bound to the Hadamard-Pini estimate involves integration along a straight line segment (a geodesic segment in the setting of Riemannian manifolds) connecting the space-time points being considered. Adjusting this argument to accommodate two paths -- one optimal for the lower bound and one optimal for the upper bound --  is not obvious, especially given the fundamental nature of the geodesic path used in the original argument. 

In this note we addressed the diffusion process driven by the ``half-Laplace'' operator and we took advantage of the explicit form of the heat kernel. The aim was to obtain, by a direct computation, a sharp Harnack bound, free of the artifacts brought about by the standard estimation techniques. Even though there is no such representation of fractional heat kernels for other powers of the Laplace operator,  precise asymptotic bounds have been available for a long time \cite{BLUMENTHAL_GETOOR_1960,POLYA_1923}. The exact form of the Harnack bound \eqref{FLY} sheds light on the way in which the relative position of the two space-time points manifests itself in the estimate. A generalisation of this result to other powers of the operator and arbitrary spatial dimension will be addressed in a separate paper.          
\subsection*{Acknowledgements}
This research was partially supported by the Polish National Science Center grant SONATA BIS no.\,2020/38/E/ST1/00596.

The second author also acknowledges the support of the NAWA Bekker Scholarship Programme BPN/BEK/2021/1/00277.   
\bibliographystyle{plain} 
\bibliography{biblio} 
 \end{document}